\documentclass[14pt,a4paper]{article}
\usepackage{graphicx}
\usepackage{latexsym}
\usepackage[all]{xy}
\usepackage{amsfonts}
\usepackage{amsthm}
\usepackage{amsmath}
\usepackage{amssymb}

\def\<{{\langle}}
\def\>{{\rangle}}

\def\eps{\epsilon}

\def\note#1{{}}

\def\note#1{}

\def\rend#1#2{{{\rm End}\sb{#1}(#2)}}

\def\beq{\begin{equation}}
\def\eeq{\end{equation}}

\def\id{{I}}

\def\ot{{\otimes}}

\newtheorem{thm}{Theorem}[section]

\def\Label#1{\label{#1}\ifmmode\llap{[#1] }\else 
\marginpar{\smash{\hbox{\tiny [#1]}}}\fi}
\def\Label{\label}

\newtheorem{proposition}{Proposition}[section]

\newtheorem{theorem}[proposition]{Theorem}

\theoremstyle{definition}
\newtheorem{definition}[proposition]{Definition}

\theoremstyle{remark}
\newtheorem{remark}[proposition]{Remark}

\newcounter{c}

\newcommand{\etyk}[1]{\vspace{-2.4mm}$$\begin{equation}\Label{#1}
\addtocounter{c}{1}}
\renewcommand{\]}{\ifnum \value{c}=1 $$\else \end{equation}\fi}
\begin{document}

\title{LIE ALGEBRAS AND  YANG--BAXTER EQUATIONS}

\bigskip

\author{Florin  F. Nichita\\
Institute of Mathematics "Simion Stoilow" of the Romanian Academy}

\maketitle

\bigskip

\begin{abstract}

At the previous congress (CRM 6), we 
reviewed the constructions of
Yang-Baxter operators from associative algebras,
and presented some (colored) bialgebras and Yang-Baxter systems
related to them.

The current talk deals with Yang-Baxter operators from
$(\mathbb{G},\theta)$-Lie algebras
(structures which unify 
the Lie algebras and Lie superalgebras). Thus, we 
produce solutions for the  
constant and the  spectral-parameter 
Yang-Baxter equations,
Yang-Baxter systems, etc.

Attempting to present the general framework
we review the work of other authors and
we propose problems, applications
and directions of study.

\end{abstract}

\section{Introduction}

Quantum Groups can be identified with quasitriangular Hopf Algebras. This notion is due to Drinfeld,
motivated by developments in mathematical physics. The significance of the quasitriangular condition
is that it gives an explanation for the Yang-Baxter equation (see \cite{LamRad:Int, Kas:Qua, ni}).  
This equation plays a role in Theoretical Physics (\cite{nipo}),
Knot Theory (\cite{mn}),
Quantum Groups (\cite{Nic:sel, np, np0}), etc.

In the next section, we review
the constructions of
Yang-Baxter operators from associative algebras,
the associated bialgebras and some results on Yang-Baxter systems
(from \cite{np0} and \cite{nipo2}).

Section 3
deals with Yang-Baxter operators from
$(\mathbb{G},\theta)$-Lie algebras
(structures which unify 
the Lie algebras and Lie superalgebras). We 
produce solutions for the  
constant and the  spectral-parameter 
Yang-Baxter equations and
Yang-Baxter systems (see \cite{nipo2}).

Finally, we present the general framework,
results of other authors,
and our new results. We discuss about an extension
for the duality between Lie algebras and Lie coalgebras,
Poisson algebras, and the classical Yang-Baxter equation.

In this paper we propose (open) problems, applications
and directions of study.

\section{ Non-linear equations and bialgebras}

This section is 
a survey  
on Yang-Baxter operators from algebra structures
and some related topics: connections to knot theory,
FRT constructions,
coloured Yang-Baxter operators and  Yang-Baxter
systems.

The quantum Yang-Baxter equation (QYBE) 
first appeared in theoretical physics and 
statistical mechanics.
It plays a crucial role in knot theory, in analysis
of integrable systems, in quantum and statistical mechanics and also in 
the
theory of quantum groups. 
In the quantum group theory, the solutions of
the constant QYBE lead to examples of bialgebras
via the $FRT$ construction \cite{frt, Kas:Qua}. 
Non-additive solutions of the two-parameter form of the QYBE are 
referred to as a {\em coloured
Yang-Baxter operators}.
{\em Yang--Baxter systems} (\cite{Hla:alg, 
HlaKun:qua, HlaSno:sol}) emerged from the study of quantum
integrable systems, as generalisations of the QYBE related to 
nonultralocal models.

\subsection{The constant QYBE}

Throughout this paper $ k $ is a field. All tensor products appearing in this paper are defined over $k$.
For $ V $ a $ k$-space, we denote by
$ \   \tau : V \otimes V \rightarrow V \ot V \  $ the twist map defined by $ \tau (v \ot w) = w \ot v $, and by $ I: V \rightarrow V $
the identity map of the space V.

We use the following notations concerning the Yang-Baxter equation.

If $ \  R: V \ot V \rightarrow V \ot V  $
is a $ k$-linear map, then
$ {R^{12}}= R \ot I , {R^{23}}= I \ot R ,
{R^{13}}=(I\ot\tau )(R\ot I)(I\ot \tau ) $.

\begin{definition}
An invertible  $ k$-linear map  $ R : V \ot V \rightarrow V \ot V $
is called a Yang-Baxter
operator if it satisfies the  equation
\begin{equation}  \label{ybeq}
R^{12}  \circ  R^{23}  \circ  R^{12} = R^{23}  \circ  R^{12}  \circ  R^{23}
\end{equation}
\end{definition}
\bigskip
\begin{remark}
The equation (\ref{ybeq}) is usually called the braid equation. It is a
well-known fact that the operator $R$ satisfies (\ref{ybeq}) if and only if
$R\circ \tau  $ satisfies
  the constant QYBE 
(if and only if
$ \tau \circ R $ satisfies
the constant QYBE):
\begin{equation}   \label{ybeq2}
R^{12}  \circ  R^{13}  \circ  R^{23} = R^{23}  \circ  R^{13}  \circ  R^{12}
\end{equation}

\end{remark}
\bigskip
\begin{remark}
(i) An exhaustive list of invertible solutions for (\ref{ybeq2}) in dimension 
2 is given in \cite{hi}.

(ii) Finding all Yang-Baxter operators in dimension greater than 2 is an 
unsolved problem.

\end{remark}

\bigskip

Let $A$ be an associative $k$-algebra, and $ \alpha, \beta, \gamma \in k$. 
We define the
$k$-linear map:\\
$ R^{A}_{\alpha, \beta, \gamma}: A \ot A \rightarrow A \ot A, \ \ 
R^{A}_{\alpha, \beta, \gamma}( a \ot b) = \alpha ab \ot 1 + \beta 1 \ot ab -
\gamma a \ot b $.

\begin{theorem} (S. D\u asc\u alescu and F. F. Nichita, \cite{DasNic:yan})\label{primat}
Let $A$ be an associative 
$k$-algebra with $ \dim A \ge 2$, and $ \alpha, \beta, \gamma \in k$. Then $ R^{A}_{\alpha, \beta, \gamma}$ is a Yang-Baxter operator if and only if one
of the following holds:

(i) $ \alpha = \gamma \ne 0, \ \ \beta \ne 0 $;

(ii) $ \beta = \gamma \ne 0, \ \ \alpha \ne 0 $;

(iii) $ \alpha = \beta = 0, \ \ \gamma \ne 0 $.

If so, we have $ ( R^{A}_{\alpha, \beta, \gamma})^{-1} = 
R^{A}_{\frac{1}{ \beta}, \frac{1}{\alpha}, \frac{1}{\gamma}} $ in cases (i) and
(ii), and $ ( R^{A}_{0, 0, \gamma})^{-1} = 
R^{A}_{0, 0, \frac{1}{\gamma}} $ in case (iii).
\end{theorem}

\bigskip

\begin{remark}

The Yang--Baxter equation plays an important role in knot theory. 
Turaev has described a general scheme to derive an invariant of 
oriented links from a Yang--Baxter operator, provided this 
one can be ``enhanced''.
In \cite{mn}, we considered the problem of applying Turaev's method to the
Yang--Baxter operators derived from algebra structures presented in
the above theorem. 

We concluded that the Alexander polynomial is the knot invariant corresponding 
to the axioms of associative algebras. 

\end{remark}

\bigskip

\begin{remark} In dimension two, the Theorem \ref{primat} leads to
the following R-matrix:
\begin{equation} \label{rmatcon2}
\begin{pmatrix}
1 & 0 & 0 & 0\\
0 & 1 & 0 & 0\\
0 & 1-q  & q & 0\\
\eta & 0 & 0 & -q
\end{pmatrix}
\end{equation}
where $ \eta \in \{ 0, \ 1 \} $, and $q \in k - \{ 0 \}$.

The FRT bialgebras associated to (\ref{rmatcon2}) 
have the following independent commutation relations:\\
(i) the case $ \eta =0 $
$$ba = qab, ac=ca, [a,d] = (1-q)cb, (1+q)b^2=0,$$
$$bc = qcb, bd = -qdb, (1+q)c^2 = 0, dc=-cd $$
(ii) the case $ \eta =1 $
$$ba = qab, 
ab = dc + cd,
[a,c] = db, 
a^2 - d^2 = (1+q)c^2,$$
$$ [a,d]=(1-q)cb, b^2=0,
bc = qcb, bd=-qdb $$
where $[a,c] = ac - ca, \ \  [a,d] = ad - da $.

The comultiplication
$ \delta (T) = T \otimes T $
and counit
$ \eps (T) = I_2 $
form the underlying coalgebra structure, where
$
T=\begin{pmatrix}
a & b\\
c & d
\end{pmatrix}
$
and
$
I_2=\begin{pmatrix}
1 & 0\\
0 & 1
\end{pmatrix}
$. The coquasitriangular structure is associated in the standard way.

\end{remark}

\subsection{The two-parameter form of the QYBE}

Formally, a coloured Yang-Baxter operator is defined as a function $$ R
:X\times X \to \rend k {V\otimes V}, $$ where $X$ is a set and $V$ is a
finite dimensional vector space over a field $k$. 

Thus, for any $u,v\in X$,
$R(u,v) : V\otimes V\to V\otimes V$ is a linear operator. 

We consider three operators acting on a triple
tensor product $V\otimes V\otimes V$, $R^{12}(u,v) = R(u,v)\otimes
\id$, $R^{23}(v,w)= \id\otimes R(v,w)$, and similarly $R^{13}(u,w)$ as
an operator that acts non-trivially on the first and third factor in
$V\otimes V\otimes V$. 

$R$ is a coloured Yang-Baxter operator if it
satisfies the two-parameter form of the QYBE,
\begin{equation}\label{yb} 
R^{12}(u,v)R^{13}(u,w)R^{23}(v,w) = R^{23}(v,w)
R^{13}(u,w)R^{12}(u,v)
\end{equation} 
for all $u,v,w\in X$. 


\bigskip

Below, we present families of solutions for the equations (\ref{yb}).
We assume that $X$ is equal to (a subset of) the ground field $k$.
The key point of the construction is to suppose 
that $V=A$ is an associative 
$k$-algebra, and then to derive a solution to equation (\ref{yb}) from 
the associativity of the product in $A$. 

Thus, in \cite{np}, we sought solutions to equation (\ref{yb}) of the following form
\begin{equation}\label{rans}
R(u,v)(a\otimes b) =\alpha(u,v)1\otimes ab + \beta(u,v)ab\otimes 1 -\gamma(u,v)b\otimes a,
\end{equation}
where $\alpha, \beta,\gamma$ are $k$-valued functions on $X\times X$. 

Inserting this ansatz into equation (\ref{yb}), 
we obtained the following system of equations
(whose solutions produce coloured Yang-Baxter operators): 
\begin{eqnarray} &&
(\beta(v,w)-\gamma(v,w))(\alpha(u,v)\beta(u,w) -
\alpha(u,w)\beta(u,v))\nonumber \\ &&\quad \quad \quad +
(\alpha(u,v)-\gamma(u,v))(\alpha(v,w)\beta(u,w) - \alpha(u,w)\beta(v,w))
= 0 \label{e1} \\ \nonumber \\ &&
\beta(v,w)(\beta(u,v)-\gamma(u,v))(\alpha(u,w)-\gamma(u,w)) \nonumber \\
&&\quad \quad \quad +
(\alpha(v,w)-\gamma(v,w))(\beta(u,w)\gamma(u,v)-\beta(u,v)\gamma(u,w)) =
0 \label{e2} \\ \nonumber \\ && \alpha(u,v)
\beta(v,w)(\alpha(u,w)-\gamma(u,w)) + \alpha(v,w)\gamma(u,w)
(\gamma(u,v) - \alpha(u,v)) \nonumber \\ &&\quad \quad \quad +
\gamma(v,w) (\alpha(u,v)\gamma(u,w)-\alpha(u,w)\gamma(u,v)) = 0
\label{e3} \\ \nonumber \\ && \alpha(u,v)
\beta(v,w)(\beta(u,w)-\gamma(u,w)) + \beta(v,w)\gamma(u,w) (\gamma(u,v)
- \beta(u,v)) \nonumber \\ &&\quad \quad \quad + \gamma(v,w)
(\beta(u,v)\gamma(u,w)-\beta(u,w)\gamma(u,v)) = 0 \label{e4} \\
\nonumber \\ && \alpha(u,v)(\alpha(v,w)-\gamma(v,w))(\beta(u,w) -
\gamma(u,w)) \nonumber \\ &&\quad \quad \quad + (\beta(u,v)-
\gamma(u,v))( \alpha(u,w) \gamma(v,w) - \alpha(v,w) \gamma(u,w)) = 0 \
\label{e5} \end{eqnarray}

\begin{remark}
(i) The system of equations (\ref{e1}--\ref{e5}) is rather non-trivial. It 
is an open problem to classify its solutions. 

(ii) We found the following solutions to that system:
$ \  \alpha(u,v)= p(u-v) $, $ \ \beta(u,v)=q(u-v) \ $ and 
$ \gamma (u,v)= pu-qv $.
(Thus, we obtained the next theorem.)

\end{remark}

\begin{theorem} (F. F. Nichita and D. Parashar, \cite{np})
For any two parameters $p,q\in k$, the function
$R:X\times X\to \rend k {A\otimes A}$ defined by
\begin{equation}\label{rsol} 
R(u,v)(a\otimes b) =p(u-v)1\otimes ab + q(u-v)ab\otimes 1 -(pu-qv)b\otimes a,
\end{equation}
satisfies the coloured QYBE (\ref{yb}).
\end{theorem}

\begin{remark}
If $ \ pu \neq qv $ and $ \ qu \neq pv $ then the operator
($\ref{rsol}$) is invertible. Moreover, the following formula holds:
$$ R^{-1}(u,v)(a\otimes b) = \frac{p(u-v)}{(qu-pv)(pu-qv)}ba\otimes 1 + 
\frac{q(u-v)}{(qu-pv)(pu-qv)}1\otimes ba - \frac{1}{(pu-qv)}b\otimes a $$

\end{remark}

\bigskip

Algebraic manipulations of the previous theorem lead to the following result.

\begin{theorem} (F. F. Nichita and B. P. Popovici, \cite{nipo2}) \label{top}
Let $A$ be an associative 
$k$-algebra with $ \dim A \ge 2$ and
 $q\in k$. Then the operator
\begin{equation}\label{slsol} 
S( \lambda )(a\otimes b) = (e^{\lambda} - 1)1\otimes ab 
+ q(e^{\lambda} - 1)ab\otimes 1 -(e^{\lambda}-q)b\otimes a
\end{equation}
satisfies the one-parameter form of the Yang-Baxter equation:
$$S^{12} (\lambda_{1} - \lambda_{2}) S^{13} (\lambda_{1} - \lambda_{3}) S^{23}(\lambda_{2} - \lambda_{3})=$$
\begin{equation}\label{onepara}
= S^{23} (\lambda_{2} - \lambda_{3}) S^{13} (\lambda_{1} - \lambda_{2})
S^{12} (\lambda_{1} - \lambda_{2}).
\end{equation}
If $ \ e^{\lambda} \neq q  , \ \frac{1}{q} $, $ \ $ then the operator
($\ref{slsol}$) is invertible.
 
Moreover, the following formula holds:

$ \ \ \ S^{-1}(\lambda)(a\otimes b) = \frac{e^{\lambda}-1}{(qe^{\lambda}-1)(e^{\lambda}-q)}ba\otimes 1 + 
\frac{q(e^{\lambda}-1)}{(qe^{\lambda}-1)(e^{\lambda}-q)}1\otimes ba 
- \frac{1}{e^{\lambda}-q}b\otimes a $.
\end{theorem}

\begin{remark}

The operator from Theorem \ref{top} can be obtained from
Theorem \ref{primat} and the {\bf Baxterization} procedure from \cite{defk} (page 22). 

Hint: Consider the operator
$ \  R^{A}_{q, \frac{1}{q}, \frac{1}{q} }: A \ot A \rightarrow A \ot A, \ \ 
 a \ot b \mapsto q ab \ot 1 +  \frac{1}{q} \ot ab -
 \frac{1}{q} a \ot b $ and its inverse, $ R^{A}_{q, \frac{1}{q}, q } $.

\end{remark}

\bigskip

\subsection{Yang-Baxter systems}

It is convenient to describe the Yang-Baxter systems in terms of
the Yang-Baxter commutators.

Let $V$, $V'$, $V''$ be finite dimensional
vector spaces over the field $k$, and let $R: V\ot
V' \rightarrow V\ot V'$, $S: V\ot V'' \rightarrow V\ot V''$ and $T:
V'\ot V'' \rightarrow V'\ot V''$ be three linear maps.
The {\em constant Yang--Baxter
commutator} is a map $[R,S,T]: V\ot V'\ot V'' \rightarrow V\ot V'\ot
V''$ defined by \beq [R,S,T]:= R^{12} S^{13} T^{23} - T^{23} S^{13}
R^{12}. \eeq 
Note that $[R,R,R] = 0$ is just a short-hand
notation for the constant QYBE (\ref{ybeq2}).

A system of linear
maps
$W: V\ot V\ \rightarrow V\ot V,\quad Z: V'\ot V'\ \rightarrow V'\ot
V',\quad X: V\ot V'\ \rightarrow V\ot V',$ is called a
$WXZ$--system if the
following conditions hold: \beq \label{ybsdoub} [W,W,W] = 0 \qquad
[Z,Z,Z] = 0 \qquad [W,X,X] = 0 \qquad [X,X,Z] = 0\eeq 
It
was observed that $WXZ$--systems with invertible $W,X$ and $Z$ can
be used to construct dually paired bialgebras of the FRT type
leading to quantum doubles. The above is one type of a constant
Yang--Baxter system that has recently been studied in \cite{np} and
also shown to be closely related to entwining structures \cite{bn}.

\bigskip

\begin{theorem} (F. F. Nichita and D. Parashar, \cite{np})
Let $A$ be a $k$-algebra, and $ \lambda, \mu \in k$. The following is a 
$WXZ$--system:

$ W : A \ot A \rightarrow A \ot A, \ \ 
W(a \ot b)= ab \ot 1 + \lambda 1 \ot ab - b \ot a $,

$ Z : A \ot A \rightarrow A \ot A, \ \ 
Z(a \ot b)= \mu ab \ot 1 +  1 \ot ab - b \ot a $,

$ X : A \ot A \rightarrow A \ot A, \ \ 
X(a \ot b)= ab \ot 1 +  1 \ot ab - b \ot a $.

\end{theorem}

\begin{remark}
Let $R$ be a 
solution for the two-parameter form of the QYBE, i.e.
$ \ R^{12}(u,v)R^{13}(u,w)R^{23}(v,w) = R^{23}(v,w)
R^{13}(u,w)R^{12}(u,v)
 \ \  \forall \ u,v,w\in X$.

Then, if we fix $ s, t \in X$, we obtain the following
$WXZ$--system:

$W= R(s, s) $,
$ \ X= R(s, t) $ and
$\ Z= R(t, t) $.

\end{remark}

\begin{remark} The Section 5  of \cite{np0}
provides connections
between the constant and coloured Yang-Baxter operators and Yang-Baxter 
systems 
from algebra structures,
which
were
discovered while presenting the poster \cite{np3} in
Cambridge (2006).
\end{remark}

\section{YANG--BAXTER OPERATORS FROM $( \mathbb{G},\theta )$-LIE ALGEBRAS}

The $(\mathbb{G},\theta)$-Lie algebras are structures which unify 
the Lie algebras and Lie superalgebras. We 
use them to produce solutions for the  
quantum Yang--Baxter equation. 
The  spectral-parameter 
Yang-Baxter equations and
Yang-Baxter systems are also studied. 
The following authors constructed 
Yang-Baxter
operators from Lie (co)algebras and Lie superalgebras before:
 \cite{mj}, \cite{ba}, \cite{ni}, etc.

\subsection{Lie superalgebras}

\bigskip
\begin{definition}
A Lie superalgebra is a (nonassociative) $\mathbb{Z}_2$-graded algebra, or superalgebra, 
over a field $k$ with the  Lie superbracket, satisfying the two conditions:
$$[x,y] = -(-1)^{|x||y|}[y,x]$$
$$ (-1)^{|z||x|}[x,[y,z]]+(-1)^{|x||y|}[y,[z,x]]+(-1)^{|y||z|}[z,[x,y]]=0 $$
where $x$, $y$ and $z$ are pure in the $\mathbb{Z}_2$-grading. Here, $|x|$ denotes the degree of $x$ (either 0 or 1). 
The degree of $[x,y]$ is the sum of degree of $x$ and $y$ modulo $2$.
\end{definition}

Let $ ( L , [,] )$ be a Lie superalgebra over $k$,
and  $ Z(L) = \{ z \in L : [z,x]=0 \ \ \forall \ x \in L \} $.

For $ z \in Z(L), \ \vert z \vert =0 $ and $ \alpha \in k $ we define:

$$ { \phi }^L_{ \alpha} \ : \ L \ot L \ \ \longrightarrow \ \  L \ot L $$

$$ 
x \ot y \mapsto \alpha [x,y] \ot z + (-1)^{ \vert x \vert \vert y \vert } y \ot x \ . $$

Its inverse is:

$$ {{ \phi }^L_{ \alpha}}^{-1} \ : \ L \ot L \ \ 
\longrightarrow \ \  L \ot L $$

$$x \otimes y \mapsto \alpha z \otimes [x, y] + (-1)^{ \vert x \vert \vert y \vert 
} y \otimes x$$

\begin{theorem} (F. F. Nichita and B. P. Popovici, \cite{nipo2})

Let  $ ( L , [,] )$ be a Lie superalgebra 
and
$ z \in Z(L), \vert z \vert = 0  $, and $ \alpha \in k $. Then:
$ \ \ \ \  { \phi }^L_{ \alpha} $ is a YB operator.
\end{theorem}

\begin{theorem} (F. F. Nichita and B. P. Popovici, \cite{nipo2})
 Let  $ ( L , [,] )$ be a Lie superalgebra, 
$ z \in Z(L), \vert z \vert = 0  $, 
$ X \subset k $,
and $ \alpha, \beta:X \times X \rightarrow k $. 
Then,
$R:X\times X\to \rend k {L \otimes L}$ defined by
\begin{equation}\label{Lsol} 
R(u,v)(a\otimes b) = \alpha(u,v)[a,b]\otimes z + \beta (u,v) (-1)^{|a||b|} a\otimes b,
\end{equation}
satisfies the colored QYBE (\ref{yb}) $ \iff 
 {\beta(u,w)} {\alpha(v,w)} =  {\alpha(u,w)} {\beta(v,w)} .$

\end{theorem}

\begin{remark}
$ \alpha(u,v)= f(v)$ and $ \beta(u,v)= g(v)$ is a solution for the above condition.

Letting $ u=v $, we obtain that:
$$ { \phi }^L_{ \alpha, \beta} \ : \ L \ot L \ \ \longrightarrow \ \  L \ot L $$

$$ 
x \ot y \mapsto \alpha [x,y] \ot z + (-1)^{ \vert x \vert \vert y \vert } \beta y \ot x \ . $$

and its inverse:

$$ {{ \phi }^L_{ \alpha, \beta}}^{-1} \ : \ L \ot L \ \ 
\longrightarrow \ \  L \ot L $$

$$x \otimes y \mapsto \frac{\alpha}{ {\beta}^2} z \otimes [x, y] + 
(-1)^{ \vert x \vert \vert y \vert 
} \frac{1}{ \beta } y \otimes x$$

are Yang-Baxter operators.
\end{remark}

\begin{remark}
Let us consider the above data and apply it to Remark 2.13.
Then, if we let $ s, t \in X$, we obtain the following
$WXZ$--system:

$W(a\otimes b) = R(s,s)(a\otimes b)=
 f(s)[a,b]\otimes z + g(s) (-1)^{|a||b|} a\otimes b, $ and

$Z(a\otimes b) = R(t,t)(a\otimes b)=   \ X(a\otimes b)= R(s,t)(a\otimes b)=
f(t)[a,b]\otimes z + g(t) (-1)^{|a||b|} a\otimes b$.

\end{remark}

\begin{remark}
The results presented in this section hold for Lie algebras as well.
This is a consequence of the fact that these operators restricted to the
first component of a Lie superalgebra have the same properties.

\end{remark}

\subsection{ $(\mathbb{G},\theta)$-Lie algebras}

We now consider the case of  $(\mathbb{G},\theta)$-Lie algebras as in \cite{Kanak}: a generalization of
Lie algebras and Lie superalgebras. 

A $(\mathbb{G},\theta)$-Lie algebras consists of a $\mathbb{G}$-graded 
vector space $L$, with $L=\oplus_{g\in\mathbb{G}}L_g$,  $\mathbb{G}$ a finite abelian group, a non associative 
multiplication $\langle ..,..\rangle : L \times L \to L$ respecting the graduation in the sense that
$\langle L_a,L_b\rangle \subseteq  L_{a+b}, \;\; \forall a,b\in \mathbb{G}$ and a function 
$\theta:\mathbb{G}\times\mathbb{G}\to C^{*} $ taking non-zero complex values. The following conditions
are imposed:
\begin{itemize}
  \item $\theta$-braided (G-graded) antisymmetry: $\langle x,y\rangle = -\theta(a,b)\langle y,x  \rangle$ 
  \item $\theta$-braided (G-graded) Jacobi id: $\theta(c,a)\langle x, \langle y,z\rangle\rangle + \theta(b,c)\langle z, \langle x,y\rangle\rangle +\theta(a,b)\langle y, \langle z,x\rangle\rangle =0 $
  \item $\theta : G \times G \to C^*$ color function 
$\left \{ \begin{array}{c}\theta(a+b,c) = 
\theta(a,c)\theta(b,c)\\ 
\theta(a,b+c) = \theta(a,b)\theta(a,c)\\   
\theta(a,b)\theta(b,a) = 1 \end{array}\right . $
\end{itemize}  
for all homogeneous $x\in L_a, y \in L_b, z\in L_c$ and $\forall a,b,c \in \mathbb{G}$.

\begin{theorem} (F. F. Nichita and B. P. Popovici, \cite{nipo2})
Under the above assumptions,

\begin{equation}\label{Lsol} 
R(x\otimes y) = \alpha[x,y]\otimes z + 
\theta(a,b) x\otimes y, 
\end{equation}
with $ z \in Z(L)$,
satisfies the equation ( \ref{ybeq2} )
$ \iff 
\theta(g,a)= \theta(a,g)=\theta(g,g)=1$, $\forall x\in L_a$ and $z\in L_g$.

The inverse operator reads: $R^{-1}(x\otimes y) = 
\alpha [y,x] \otimes z + \theta(b,a) x\otimes y $
\end{theorem}

\begin{proof} If we consider the homogeneous elements $x\in L_a$, $y\in L_b$, $t\in L_c$, 
$$ R^{12} R^{13}R^{23} (x \ot y \ot t) = R^{23} R^{13} R^{12} (x \ot y \ot t)  $$
is equivalent to
\begin{eqnarray}
\theta(a,g)[x,[y,t]]\ot z \ot z +  \theta(b,c) [[x,t],y]\ot z \ot z = \theta(g,g)  [[x,y],c]\ot z \ot z   \\
\theta(a,g)\theta(a,b+c) x\ot [y,t] \ot z = \theta(a,b)\theta(a,c)x\ot\ [y,t] \ot z    \\
\theta(b,c) \theta(a+c,b) [x,t]\ot y \ot z = \theta(a,b)\theta(b,g)[x,t]\ot y \ot z  \\   
\theta(b,c) \theta(a,c) [x,y] \ot z \ot t  = \theta(a+b,c) \theta(g,c)[x,y]\ot z \ot t 
\end{eqnarray}

Due to the conditions $\langle L_a, L_b\rangle \subseteq L_{a+b} $ the above relations   are true 
if $\theta(a,g)=\theta(b,g)=\theta(g,c)=\theta(g,g)=1$ is assumed.
\end{proof}

\section{Applications. Problems. Directions of study}

\subsection{A Duality Theorem for (Co)Algebras}

\ \ \ \ \ \ Our aim in
this subsection is to present an extension of the duality of finite dimensional algebras and
coalgebras to the category of finite dimensional Yang-Baxter structures, denoted {\bf 
f.d. YB \ str}.

\bigskip

\begin{definition}
We define the category ${\bf YB \ str \,\ }$(respective ${\bf f.d. YB \ str}$)
whose objects are 4-tuples $ (V,\,\varphi ,\,e,\,\varepsilon )\,$,
where\\
i) \ \ \ \ $V$ is a (finite dimensional) $k$-space;\\
$ii)\ \ \ \ \ \varphi :V\otimes V\rightarrow V\otimes V$ \ is a YB
operator;\\
iii)$\ \ \ \ \ \ e\in V\,\ \ $ such that $\ \ \varphi (x\otimes e)=e\otimes
x,\,\ \varphi (e\otimes x)=x\otimes e \ \ \forall x \in V \ $;\\
iv) $\ \ \varepsilon \in V\rightarrow k\ $ is\ a $k$-map such
that $ \ \ \ (I\,\otimes \varepsilon )\circ \varphi =\varepsilon \otimes
I,\,\ \ \ (\varepsilon \otimes I)\circ \varphi =I\,\otimes \varepsilon \,.$

A morphism $f:(V,\,\varphi ,\,e,\,\varepsilon )\rightarrow (V^{\prime
},\,\varphi ^{\prime },\,e^{\prime },\,\varepsilon ^{\prime })$ in the
category ${\bf YB \ str}$
is a $k$-linear map $f:V\rightarrow V^{\prime }$ such that:\\
v) $\ \ \ \ (f\otimes f)\circ \varphi \,=\,\varphi ^{\prime }\circ
(f\otimes f) \ ; $\\
vi) $ \ \ \ \ f(e)=e^{\prime } \ ; $\\
vii) $ \ \ \ \ \varepsilon ^{\prime }\circ f=\varepsilon \,.$
\end{definition}

\vspace{0.2cm}

\begin{remark} The following are examples of objects from the category ${\bf
YB \ str}$:

(i) Let $\ R:V\otimes V\rightarrow V\otimes V$ \ is a YB operator.
Then $(V,\,R,\,0,\,0)\,\ $is an object in the category ${\bf YB \ str}$.

(ii)  Let $V\,\ $be a two dimensional $k$-space generated by the vectors $\
e_{1} \ $ and $ \ e_{2} \ $.
Then $(V,\,T,e_{1}\,,e_{2}^{\ast }\,)\,\ $is
an object in the category ${\bf f.d. \ YB \ str}$.

\end{remark}

\begin{thm} (F. F. Nichita and S. D. Schack, \cite{ns})
i) \ There exists a functor:

$\ \ F:{\bf k-alg}\longrightarrow \ {\bf YB \ str}$

$(\ A,\ M,\ u)\mapsto (A,{\varphi }_{A},\,u(1)=1_{A},\,0\in A^{\ast })\ \ \
\ \ \ where\ \ {\varphi }_{A}(a\otimes b)=ab\otimes 1+1\otimes ab-a\otimes
b. $

  Any k-algebra map
f is simply mapped into a k-map.

ii) $\ F$ is a full and faithful embbeding.
\end{thm}

\begin{thm} (F. F. Nichita and S. D. Schack, \cite{ns})
i) There exists a functor:

$G:{\bf k-coalg}\ \longrightarrow \ {\bf YB \ str}$

$(C,\Delta ,\varepsilon )\mapsto (C,\,{\psi }_{C},\,0\in C,\,\varepsilon
\in
C^{\ast })\ \ \ \ \ \ where\ \ {\psi }_{C}=\Delta \otimes \varepsilon
+\varepsilon \otimes \Delta -I_{2}$. \newline
\ \ \ \ \ \ \ \ \ \ \ \ \ \ \ \ \ \ \ \ \ \ \ \ \ \ \ \ \ \ Any k-coalgebra
map f is simply mapped into a k-map.

ii) $G$ is a full and faithful embbeding.
\end{thm}

\begin{thm} (F. F. Nichita and S. D. Schack, \cite{ns})
({\bf Duality Theorem})

i) \ The following is a duality functor: ${\bf \ \ \  D \ : \ f.d. \ YB \ str
\longrightarrow {f.d. \ YB \ str}^{op}}$\\
$(V,\,\varphi ,\,e,\,\varepsilon )\mapsto (V^{\ast },\,i_{V,V}^{-1}\circ
\varphi ^{\ast }\circ {i}_{V,V},\,\varepsilon ,\,\zeta _{e})$ \ where $\,\
\zeta _{e}:V^{\ast }\rightarrow k,\,\,\,\zeta _{e}(g)=g(e) \ \ \ 
\forall g\in V^{\ast }.$\\
\ \ Note that: $\ \ \ \ \ \ \ \ \ \ D(f)=f^{\ast }$, \ \ for \
$f:(V,\,\varphi
,\,e,\,\varepsilon )\rightarrow (V^{\prime },\,\varphi ^{\prime
},\,e^{\prime },\,\varepsilon ^{\prime }).$ 

ii) \ The following relations hold:\\
$ \ \ \ \ D(\,(A,\,{\varphi }_{A},\,1_{A},\,0)\,) \ \ = \ \ (A^{\ast },\,\psi _{A^{\ast
}},\,0,\,\zeta _{1_{A}}) $\\
$ \ \ \ \ D(\,(C,\,{\psi }_{C},\,0,\,\varepsilon
)\,)\,=\,(C^{\ast },\,\varphi _{C^{\ast }},\,\varepsilon =1_{C^{\ast
}},\,0)$
\end{thm}

\vspace{.2cm}


\begin{remark} We extended the duality between finite dimensional 
algebras and coalgebras to the category  $ {\bf f.d.
  \ YB \ str} $. This can be seen bellow, in the following diagram:

\begin{picture}(100,100)(10,10)
\put(60,80){$ \bf f.d. \ YB \ str $ }
\put(70,10){$ \bf f.d. \ k-alg $ }
\put(111,21){\vector(0,1){53}}
\put(240,80){$ \bf {f.d. \ YB \ str}^{opp} $ }
\put(230,10){$ \bf {f.d. \ k-coalg}^{opp} $ }
\put(250,21){\vector(0,1){53}}
\put(128,86){\vector(1,0){104}}
\put(232,80){\vector(-1,0){104}}
\put(137,16){\vector(1,0){87}}
\put(223,10){\vector(-1,0){87}}
\put(162,92){$ D = { ( \ ) }^* $}
\put(162,70){$ D = { ( \ ) }^* $}
\put(95,42){ $ F $ }
\put(248,42){ $ G $ }
\put(174,22){$ { ( \ ) }^*   $}
\put(174,0){$ { ( \ ) }^*   $}
\end{picture}

\end{remark}

\subsection{A Duality Theorem for Lie (Co)Algebras}

\cite{dns} considered the constructions of Yang-Baxter operators
from Lie (co)algebras, suggesting  an extension
(to a bigger category with a self-dual functor acting on it) for the
duality between
the category of finite dimensional  Lie algebras 
and
the category of finite dimensional  Lie coalgebras. This duality extension
is explained using the terminology of \cite{ni} below.

\bigskip

Let $ ( L , [,] )$ be a Lie algebra over $k$. Then we can equip 
$ \ L'\ = \ L \oplus k x_0  $ with a Lie algebra structure such that
$ \ [x, \ x_0] = 0 \ \ \forall x \in L'$.
We define:
$$ { \phi }=  {\phi }_{L'}  \ : \ (L \oplus k x_0 ) \ot (L \oplus k x_0 ) \ \ \longrightarrow \ \ (L \oplus k x_0 )  \ot (L \oplus k x_0 )  $$
$$ 
x \ot y \mapsto [x,y] \ot x_0 +  y \ot x \ . $$

\begin{thm} 
i) \ There exists a functor:

$\ \ F:{\bf f.d.\ Lie \  alg}\longrightarrow \ {\bf f.d.\ YB \ str}$

$(L,\ [,])\mapsto ((L \oplus k x_0 ) ,{\phi },\ x_0 , \, 0)$.

Any Lie algebra map
f is simply mapped into a k-map.

ii) $F$ is a full and faithful embbeding.

\end{thm}

\bigskip {\bf Proof:}

\ \ \ \ \ \ \ \ \ \ i) \ First, we show that $\ (L', \ {\phi 
}_{L'},x_{0},\,0)\,\ $is an
object in the category ${\bf YB \ str} :$

$\ \  \ {\phi }_{L'}(x\otimes x_{0}) =  x_{0} \ot x$,
$\ \  \ {\phi }_{L'}(x_0 \otimes x) =  x \ot x_0$,

$ \ \ \  (I\,\otimes 0)\circ \phi _{L'}=0=0\otimes I, \ \
\ \ \ (0\otimes I)\circ \phi _{L'}=0=I\,\otimes 0\,.$

Now, for \ \ $f:L_1\rightarrow L_2\,\ \,\ \ $a
morphism of Lie algebras, we prove that

$  \ \ \ \ \ f:(L'_1,\,\phi_{L'_1, },\,x_0,\,0)\rightarrow (L'_2, \ \phi _{L'_2}, x_0, \ 0)$ \ is a
morphism in the category ${\bf
YB \ str } 
$.

We extend $f$ such that $f(x_0)=x_0 \ $. Now, $\ 0\circ f=0\,\ $. It only
remains to
prove that $(f\otimes f)\circ \phi _{L'_1}\,=\,\phi _{L'_2}\circ (f\otimes
f).$

$( (f\otimes f)\circ \phi _{L'_1}) \,(x\otimes
y)=(f\otimes f)([x,y]\otimes x_{0}+y\otimes x)=f([x,y])\otimes
f(x_{0})+f(y)\otimes f(x)=
f([x,y])\otimes
x_{0}+f(y)\otimes f(x)$

$( \phi _{L'_2}\circ (f\otimes
f))(x \ot y)=[f(x),f(y)]\otimes
x_{0}+f(y)\otimes f(x) .$

Since 
$f:L_1\rightarrow L_2\,\ \,\ \ $ is a
morphism of Lie algebras, 
it follows that $(f\otimes f)\circ \phi _{L'_1}\,=\,\phi _{L'_2}\circ (f\otimes
f).$


\ \ \ \ \ \ \ \ \ ii) \ If two Lie algebras $\ \ (L_1,\ [,]_1 ) \ $and $\
(L_2, \ [,]_2 ) \ $ project in the same object in the category
$ { \bf YB \ str } $ (i.e., $  \  F \left[ (L_1,\ [,]_1 ) \right] =  
F \left[ (L_2,\ [,]_2 ) \right] $) 
then they have the
same
ground vector space and the same operation.
So, \ $F$\ \ \ is an embedding.

Obviously, for two distinct Lie algebra maps $\
f,g:L_1 \rightarrow L_2 \ $ we get two distinct ${\bf YB \ str\,\ }$ maps.

Now, for 
 $  \  f:(L'_1,\,\phi_{L'_1, },\,x_0,\,0)\rightarrow (L'_2, \ \phi _{L'_2}, x_0, \ 0)$ a 
morphism in ${\bf YB \ str \ } $ it follows
$ ( (f\otimes f)\circ \phi _{L'_1}) \,(x\otimes
y)= (\phi _{L'_2}\circ (f\otimes
f))(x \ot y)$; so, 

$(f\otimes f)([x,y]\otimes x_{0}+y\otimes x)= [f(x),f(y)]\otimes
x_{0}+f(y)\otimes f(x)$.

Thus, 
$ f([x,y]) = [f(x), f(y)]$.
\qed

\bigskip

A Lie coalgebra is a dual notion to a Lie algebra. It has a comultiplication,
called ``cobraket''.
We refer to \cite{ni} for more details and references.
 
Let $ ( M , \Delta )$ be a Lie coalgebra over $k$. Then we can equip 
$ \ M'= \ M \oplus k x_0  $ with a Lie coalgebra structure such that
$ \ \Delta ( x_0) = 0 \in M' \ot M'  $. Observe that for 
$ \nu = (x_0)^*: M' \rightarrow k$ the following relation holds:
$ \ ( \nu \ot I) \circ \Delta = 0= (I \ot \nu) \circ \Delta   $.

\begin{thm}
i) \ There exists a functor:

$\ \ G:{\bf f.d. \ Lie \  coalg}\longrightarrow \ {\bf f.d. \ YB \ str}$

$(\ M,\ \Delta)\mapsto 
 ((M \oplus k x_0 ) ,{\psi },\ 0 , \, \nu)$, 
where

$ { \psi } \ : \ (M \oplus k x_0 ) \ot (M \oplus k x_0 ) \ \ \longrightarrow \ \ (M \oplus k x_0 )  \ot (M \oplus k x_0 ), \ \   $
$
x \ot y \mapsto  \Delta(x) \nu (y) +  y \ot x \  $.

  Any Lie coalgebra map
f is simply mapped into a k-map.

ii) $G$ is a full and faithful embbeding.

\end{thm}

{\bf Proof:} The proof is dual to the previous proof, and we will briefly explain only its key points.
By Theorem 5.2.1 of \cite{ni}, it follows that $ \psi $ is an Yang-Baxter operator. 

$ \ \ \  (I\,\otimes \nu)\circ \psi _{M'}= \nu \otimes I, \ \
\ \ \ (\nu \otimes I)\circ \psi _{M'}= I\,\otimes \nu \ \  $ follow from 
$ \ ( \nu \ot I) \circ \Delta = 0= (I \ot \nu) \circ \Delta   $.

The proof of 
ii) follows by direct computations.

Otherwise, it can be
viewed as a consequence of 
the Section 2 of \cite{DasNic:yan}. Thus, the theory of Yang-Baxter operators
from (Lie)algebras 
can be transfered to the Yang-Baxter operators from (Lie) coalgebras. 
\qed

\bigskip

\begin{remark} We extend the duality between finite dimensional Lie
algebras and Lie coalgebras to the category  $ {\bf f.d.
  \ YB \ str} $. 
This can be seen in the following diagram:

\begin{picture}(100,100)(10,10)
\put(60,80){$ \bf f.d. \ YB \ str $ }
\put(70,10){$ \bf f.d. \ Lie \  alg $ }
\put(111,21){\vector(0,1){53}}
\put(240,80){$ \bf {f.d. \ YB \ str}^{opp} $ }
\put(230,10){$ \bf {f.d. \ Lie \  coalg}^{opp} $ }
\put(250,21){\vector(0,1){53}}
\put(128,86){\vector(1,0){104}}
\put(232,80){\vector(-1,0){104}}
\put(137,16){\vector(1,0){87}}
\put(223,10){\vector(-1,0){87}}
\put(162,92){$ D = { ( \ ) }^* $}
\put(162,70){$ D = { ( \ ) }^* $}
\put(95,42){ $ F $ }
\put(248,42){ $ G $ }
\put(174,22){$ { ( \ ) }^*   $}
\put(174,0){$ { ( \ ) }^*   $}
\end{picture}

\end{remark}

\bigskip

\subsection{Poisson algebras}

Poison algebras appear in quantum groups, Hamiltonian mechanics,
the theory of Simplectic manifolds, etc.

\bigskip

\begin{definition}
A Poisson algebra is a vector space over $k$, $V$,
equipped with two bilinear products, $ \ *$ and
$ \{ \ , \ \}$, having the following properties:

- the product $*$  forms an associative $k$-algebra;

- the product $ \{ \ , \ \}$, called the Poisson bracket,  forms a Lie algebra;

- the Poisson bracket acts as a derivation on the product $*$, i.e.

$ \ \ \ \ \  \{ x,\ y*z \} = \{x, \ y \}*z+ y * \{x, \ z \} \ \ \ \forall x,y,z \in V. $

\end{definition}

\bigskip

{ \bf Examples.}

1. Any associative algebra with the commutator $[x, \ y] \ = \ xy -yx$ turns into a Poisson algebra.

2. For a vertex operator algebra, a certain quotient becomes a Poisson algebra.

\bigskip

\begin{remark}
A Lie algebra $ (L, \ [,])$ has a Poisson algebra structure such that the
Poisson bracket equals the associative product (i.e., $ \  [x,y] = x * y \ \ \forall x,y \in L$)
$ \  \iff  \ [x,y] \in Z(L) \ \forall x,y \in L$.

\end{remark}

\bigskip

\begin{thm}

Let $A$  be a Poisson algebra with a unity,  $1=1_A$, for the product *,  
such that
$ \{ x,\ 1_A \} = 0 \ \  \forall x \in A$.
Then, we have the following $WXZ$-system:

$W(x\otimes y) = \{ x,\ y \} \ot 1 \  + \  x \ot y$;

$X(x\otimes y) = 1 \ot  \{ x,\ y \} \  + \  x \ot y$;

$Z(x\otimes y) = 1 \ot  x*y \  + \  x*y \ot 1 \  - \  y \ot  x$.

\end{thm}

{ \bf Proof.} It follows by direct computations.

\bigskip

\subsection{Other results and comments}

Motivated by the need to create a better frame 
for the study of Lie (super)algebras
than that presented in \cite{tbh}, 
this paper generalizes the constructions from \cite{mj}
(to $(G,\theta)$-Lie algebras). 
Other applications of these results could be in constructions of: FRT
bialgebras and knot invariants, 
solutions for the
classical Yang-Baxter equation (see below), etc. 

\begin{theorem} Let  $ ( L , [,] )$ be a Lie algebra,
$ z \in Z(L)$ and $ \alpha \in k$. Then:

$ r: L \ot L \ \ \longrightarrow \ \  L \ot L, \ \  
x \ot y \mapsto [x,y] \ot z + \alpha x \ot y $ 

satisfies the classical Yang-Baxter equation:

$ [ r^{12},\  r^{13} ] \  + \  [r^{12}, \  r^{23}] \ + \  [r^{13}, \ r^{23}] = 0 $.
\end{theorem} 

{ \bf Proof.} It follows by direct computations.

\bigskip


\end{document}